\documentclass{article}

\usepackage[ruled, linesnumbered, commentsnumbered]{algorithm2e}
\usepackage{bm}

\usepackage[utf8]{inputenc}
\usepackage[english]{babel}
\usepackage[T1]{fontenc}

\usepackage{amsthm}
\usepackage{mathrsfs}
\usepackage{fancyhdr}

\usepackage{hyperref}
\usepackage{bm}
\usepackage{graphicx,mathptmx,amsmath,amsfonts,amssymb}

\usepackage{authblk}
\usepackage{fullpage}

\theoremstyle{plain}
\newtheorem{theorem}{Theorem}[section]

\theoremstyle{remark}

\newtheorem{definition}[theorem]{Definition}

\newcommand{\mean}{\mathbb{E}}
\newcommand{\var}{\mathbb{V}}

\newcommand{\mode}{\mathsf{mode}}

\newcommand{\cond}{\,|\,}

\newcommand{\tr}{\text{tr}}
\newcommand{\TV}{\text{TV}}

\DeclareMathOperator{\Normal}{Normal}
\DeclareMathOperator{\Laplace}{Laplace}

\DeclareMathOperator{\Exp}{Exp}
\DeclareMathOperator{\Gam}{Gamma}

\DeclareMathOperator{\InvGam}{InvGamma}

\DeclareMathOperator{\GIG}{GIG}
\DeclareMathOperator{\RIG}{RIG}

\DeclareMathOperator{\KL}{KL} 
\DeclareMathOperator{\Besselk}{K} 
\DeclareMathOperator{\ML}{MVLaplace} 

\bmdefine{\Blambda}{\lambda}
\bmdefine{\Bnu}{\nu}
\bmdefine{\Balpha}{\alpha}
\bmdefine{\Bbeta}{\beta}
\bmdefine{\Bgamma}{\gamma}
\bmdefine{\Bdelta}{\delta}
\bmdefine{\Btheta}{\theta}
\bmdefine{\Bepsilon}{\epsilon}

\newcommand{\Br}{\mathbf{r}}
\newcommand{\Bx}{\mathbf{x}}
\newcommand{\By}{\mathbf{y}}
\newcommand{\Bz}{\mathbf{z}}
\newcommand{\Bw}{\mathbf{w}}

\newcommand{\la}{\alpha_{\lambda}}
\newcommand{\lb}{\beta_{\lambda}}
\newcommand{\nga}{\alpha_{\nu}}
\newcommand{\ngb}{\beta_{\nu}}

\newcommand{\argmin}[1]{\underset{#1}{\operatorname{arg}\operatorname{min}}\;}
\newcommand{\Id}{\text{I}}

\newcommand{\e}{\mathrm{e}}

\newcommand{\Lone}{L^{1}}
\newcommand{\Ltwo}{L^{2}}

\newcommand*{\half}[1][2]{\frac{1}{#1}}
\DeclareMathOperator{\diag}{diag}		
\newcommand*{\Nul}{{\mathcal{N}}}   

\begin{document}



\title{Bayesian Hierarchical Model of Total Variation Regularisation for Image Deblurring}


\author[1]{M. Järvenpää}
\author[2]{R. Pich\'e}
\affil[1]{Department of Mathematics, Tampere University of Technology, Tampere, Finland}
\affil[2]{Department of Automation Science and Engineering, Tampere University of Technology, Tampere, Finland}

\date{\today}

\maketitle

\begin{abstract}
A Bayesian hierarchical model for total variation regularisation is presented in this paper. All the parameters of an inverse problem, including the 'regularisation parameter', are estimated simultaneously from the data in the model. 
The model is based on the characterisation of the Laplace density prior as a scale mixture of Gaussians. With different priors on the mixture variable, other total variation like regularisations e.g.~a prior that is related to t-distribution, are also obtained. 
An approximation of the resulting posterior mean is found using a variational Bayes method. In addition, an iterative alternating sequential algorithm for computing the maximum a posteriori estimate is presented. The methods are illustrated with examples of image deblurring. Results show that the proposed model can be used for automatic edge-preserving inversion in the case of image deblurring. Despite promising results, some difficulties with the model were encountered and are subject to future work. \\





\end{abstract}

\section{Introduction} \label{sec:intro}

Total variation (TV) regularisation, initially presented in \cite{Rudin1992}, is a popular alternative for restoring `blocky' images. It penalises non-smoothness in the solution while allowing occasional `jumps'. This is useful if one wants to recover 'edges' of an image from its noisy and blurred copy. The total variation regularisation penalty term is, however, more difficult to deal with than, for example, the $\Ltwo$ norm since it is not differentiable at the origin. Bayesian approach for TV regularisation is useful because is allows a natural framework for taking prior information into account and the uncertainty of the results can be assessed. This approach also enables estimation of nuisance parameters related to the inverse problem. 

In the literature there have been several studies related to the topic. A hierarchical Bayesian model for the $\Lone$ regularisation (which is also known as Lasso in the context of regression problems) was studied in \cite{Figueiredo2003, Park2008, Kyung2010}. In these papers the penalty term is interpreted as the Laplace prior which is modeled as a Gaussian scale mixture (GSM) (\cite{Andrews1974, Kotz2001}) leading to a hierachical model that is easier to deal with. This approach is related to the model presented in this work. In \cite{Kyung2010} also, a statistical model for fused Lasso featuring two penalisation terms, one $\Lone$ and the other TV penalty, was analysed in the Bayesian setting. Laplace priors have also been considered e.g.~in compressive sensing \cite{Babacan2010} and classification problems \cite{Kaban07}. 
A fully Bayesian model of Tikhonov regularisation was studied in \cite{Jin2010}, where the variational Bayes (VB) (\cite{Bishop2006, MacKay2003}) method was used. Bayesian hierarchical models are also used in several inverse problem research projects, see for instance \cite{Calvetti2010, Calvetti2008b, Calvetti2008, Wipf2009, Nummenmaa2007, Lucka11}. Fast $\Lone$ sampling methods have been considered for example in \cite{Lucka12}.
Bayesian models of TV regularisation for image processing problems have been studied in \cite{Babacan2007, Babacan09, Chantas2010, Calvetti2008b, Calvetti2008}. There also exists several fast optimisation algorithms (see e.g. \cite{Zuo2011, GoldsteinO09}) for solving TV regularisation problems in deterministic framework. 

While several Bayesian hierarchial models of TV regularisation have been proposed, in many cases the regularisation parameter is considered unknown but fixed. Especially in a deterministic optimisation approach, auxiliary methods (and possibly some trial and error) have to be used for choosing the regularisation parameter. In our approach we consider estimating all the parameters of the inverse problem simultaneously from the data. We also study TV priors that can be presented as Gaussian scale mixtures. In \cite{Babacan2007, Babacan09} the regularisation parameter was estimated as a part of their hierarchical model as in out case but they did not use Gaussian mixtures but considered isotropic version of the TV prior and some approximations had to be used. We avoid this difficulty. Also, instead of providing only sampling based solution for the inverse problem, we consider iterative alternating sequential (IAS) method for solving the maximum a posteriori (MAP) estimate although this is not fully Bayesian approach. We also consider using Variational Bayes method to approximating the posterior. 

The idea of presenting Laplace prior as a Gaussian scale mixture encourages to try other mixing densities that, to the best of ours knowledge, have not been considered in literature. This leads to more general TV like penalties than the well-known anisotropic TV penalty. These priors have also heavy tails promoting `sparsity' which make them useful for edge-preserving image reconstruction. 
For example, Student's t-distribution is a Gaussian scale mixture with inverse gamma mixing density and using generalised inverse Gaussian  as mixing density leads to some other interesting priors. In addition, we also consider a possible alternative for isotropic TV prior using two-dimensional Laplace distribution. Similar multidimensional distribution has been studied in \cite{Eltoft2006,Kotz2001} but this distribution has not been considered to be used as a TV type prior. Furthermore, a hierarchical model for the Lasso is obtainable as a special case of our formulation although we will focus mainly on total variation regularisation in this paper.

The rest of this paper is organised as follows. In the next section we introduce the familiar linear model and the idea of TV regularisation. In Section \ref{sec:model} we present the hierarchical model, derive formulas for posterior mode and mean using IAS and VB methods. Section \ref{sec:extensions} contains discussion about certain technical details, special cases and remarks. In Section \ref{sec:examples} some image deblurring examples are considered. Finally, the summary of this work is given and possible ideas for future work are discussed in Section \ref{sec:concl}. Some results and remarks of the probability distributions appearing in this work are gathered to the appendix.

\section{Problem formulation} \label{sec:problem}

The classical method for solving a linear discrete system 
\begin{equation} \label{eq:first}
 y = Hx + \text{noise},
\end{equation}
where $H$ is a given (blurring) matrix, $x$ a vector (the image) to be solved and $y$ a given vector (the measured image), is to formulate it as an optimisation problem, in particular as a least squares problem. However, if the matrix has nontrivial nullspace, for instance, it has more columns than rows, the problem has no unique solution. Even though the matrix is theoretically non-singular as is typically the case in the image deblurring problem, the matrix can be so close to being singular that numerical problems arise in practice if naive matrix inversion is tried. This issue is typically dealt with by introducing a penalisation term and approximating the original ill-posed problem with a problem that is well-posed. 
This approach leads to the problem of solving 
\begin{equation} \label{eq:1st}
 \argmin{x} \{\| Hx-y \|^2 + \delta J(x) \},
\end{equation}
where $\delta > 0$ is a regularisation parameter and $J(x)$ is a regularisation penalty, often $\Ltwo$ or $\Lone$ norm on $x$. If $\delta = 0$ then the optimisation problem (\ref{eq:1st}) simplifies to the original least-squares formulation. 

The difficulty in the deterministic approach is that one needs to select some proper value for the regularisation parameter. If too large value is chosen, the solution is dominated by the penalty term. On the other hand, if too small value is chosen, then the problem is close the original ill-posed case. Also finding the minimum by using some optimisation algorithm can be difficult and the uncertainty in the obtained result cannot be easily assessed.

Next we mention some remarks on our notation. We will consider images with size $k \times n$ pixels and the matrix $H$ will be of size $kn \times kn$. Images are considered as $k \times n$ matrices as well but for the rest of this paper we will consider them as column-wise stacked vectors having length $N=kn$ without using any special notation. For notational simplicity indexing based on presentation of these stacked vectors as arrays is used. That is, the pixel $(i,j)$ of the image $x$ is denoted as $x_{i,j}$. Also, the following notation is used to denote horizontal and vertical differences between pixels: $\nabla_{ij}^{1}x = x_{i,j+1}-x_{i,j}$ and $\nabla_{ij}^{2}x = x_{i+1,j}-x_{i,j}$, respectively.

With the notation introduced we now introduce briefly the TV penalties in the discrete setting that are considered in this paper. The discretised two-dimensional total variation functional is called isotropic TV and it is defined by
\begin{equation}\label{eq:isotv}
 \TV_{\text{iso}} (x) = \sum_{i=1}^{k} \sum_{j=1}^{n} \sqrt{\left(\nabla_{ij}^{1}x\right)^2 + \left(\nabla_{ij}^{2}x\right)^2}.
\end{equation}
For convenience, we will assume periodic boundary conditions, that is, 
\begin{subequations}
\begin{align}
x_{i,1} &= x_{i,n+1},  \quad i=1,\ldots,k \\
x_{1,j} &= x_{k+1,j}, \quad j=1,\ldots,n . 
\end{align}
\end{subequations}
Many other boundary conditions could also be considered but in this paper we limit to these. We, however, note that our methods will be easily used with other boundary conditions as well. A simple and typical approach is to approximate the isotropic TV penalty (\ref{eq:isotv}) with
\begin{equation}\label{eq:isotv_approx}
 \TV_{\text{iso}} (x) \approx \sum_{i=1}^{k} \sum_{j=1}^{n} \sqrt{\left(\nabla_{ij}^{1}x\right)^2 + \left(\nabla_{ij}^{2}x\right)^2 + \beta},
\end{equation}
with some small $\beta > 0$. This approximation allows using gradient-based optimization methods \cite{Vogel2002}. 

Another TV variant that is sometimes used instead to (\ref{eq:isotv}) is called anisotropic TV and is defined as 
\begin{equation}\label{eq:anisotv}
 \TV_{\text{aniso}} (x) = \sum_{i=1}^{k} \sum_{j=1}^{n} \left\{ \left|\nabla_{ij}^{1}x\right| + \left|\nabla_{ij}^{2}x\right| \right\} .
\end{equation}
This $\Lone$-version is sometimes simply called the TV penalty. This anisotropic version of TV might be somewhat easier to deal with than $\TV_{\text{iso}}$, though neither version is differentiable at the origin. The anisotropic version is also not rotation invariant unlike the isotropic TV and the isotropic TV can be seen as a disretisation of the gradient of TV functional. 
A hierarchical model for isotropic TV as in (\ref{eq:isotv}) has been considered in \cite{Babacan2007} and \cite{Babacan09} but it involves some approximations. The anisotropic TV penalty, however, can be linked to Laplace distribution prior using the Gaussian scale mixture idea and thus in this study, we will instead focus on anisotropic TV. This framework also allows us to study other penalties which we will also call TV penalties (or TV priors in Bayesian setting) since they penalise the differences of neighbouring pixels. 

In the next section we will briefly introduce the Bayesian framework for solving inverse problems and then 
the hierarchical model for the optimisation problem (\ref{eq:1st}) with anisotropic TV penalty (\ref{eq:anisotv}) is discussed in detail.

\section{Bayesian Hierarchical model for TV regularisation} \label{sec:model}

We will first review the basic ideas of Bayesian inversion. In Bayesian statistics the unknowns are modelled as random variables. Let a random vector $\By$ denote the data (e.g. the blurred and noisy image) and $\Bx$ the parameters of interest (the resulting image). (Note that we denote random variables or vectors with boldface characters and fixed values with ordinary characters from now on.) Then the Bayes' law states that
\begin{equation}\label{eq:BBB}
p_{\Bx \cond \By}(x \cond y) = \frac{p_{\Bx}(x) p_{\By \cond \Bx}(y \cond x)}{p_{\By}(y)} \propto p_{\Bx}(x) p_{\By \cond \Bx}(y \cond x),
\end{equation}
where $p_{\By \cond \Bx}(y \cond x)$, the likelihood, denotes the probability of obtaining the data $\By$ if the parameters $\Bx$ were known. Probability distribution $p_{\Bx}(x)$ describes the prior information on parameters before any data is obtained. The result, the posterior distribution $p_{\Bx \cond \By}(x \cond y)$ expresses our information about the image $\Bx$ after both prior information and the data has been taken into account. The proportionality constant $p_{\By}(y) = \int p_{\Bx}(x) p_{\By \cond \Bx}(y \cond x) \mathrm{d} x$ can usually be ignored if one is only interested estimating the parameters of the inverse problem. Now if we assume that our prior also depends on some hyperparameter $\Bw$ we can write the Bayes' law as
\begin{equation}\label{eq:BBB2}
p_{\Bx, \Bw \cond \By}(x, w \cond y) \propto p_{\Bx \cond \Bw}(x \cond w) p_{\Bw}(w) p_{\By \cond \Bx}(y \cond x).
\end{equation}
If we integrate over the hyperparameter $\Bw$, which could be seen as nuisance parameter if we are not interested its particular value, in (\ref{eq:BBB2}) we obtain (\ref{eq:BBB}). We can think (\ref{eq:BBB2}) as a hierarchical model because part of the prior specification involves modelling that is on higher level. For example, we could specify a Gaussian prior for $\Bx$ and set its variance (as $\Bw$) to be Inverse Gamma distributed if we would not want to specify tight value for the variance. More information about Bayesian inversion can be found e.g.~in \cite{Kaipio2004,ohagan2004}.

Let us now consider the statistical version of (\ref{eq:1st}) that is given by the linear Gaussian observation model
\begin{equation}\label{eq:Lin}
\By \cond  (\Bx = x,\Bnu = \nu) \sim \Normal(H x, (\nu \Id)^{-1}).
\end{equation}
Here $\Bnu$ is a precision parameter and $\Id$ denotes the identity matrix. The blurred noisy image $\By$ is now modelled as random vector and, similarly,  the image $\Bx$ to estimated is a random vector. 
The two-dimensional discrete anisotropic TV prior on the coefficients $\Bx$ is 
\begin{equation}   \label{eq:priorTV}
      p_{\Bx\cond \Blambda }(x\cond \lambda)\propto \lambda^{N}
      \e^{-\sqrt{\lambda} \sum\limits_{i=1}^{k} \sum\limits_{j=1}^{n} \left\{ \left|\nabla_{ij}^{1}x \right| + \left|\nabla_{ij}^{2}x \right| \right\} }.
\end{equation}
This TV prior penalises oscillations while allowing occasional jumps. We have know connected the anisotropic TV penalty in Section \ref{sec:problem} to the prior distribution to be used in the Bayesian framework. 
The hyperparameter $\Blambda$ controls now the overall `strength' of the penalisation.

For conjugacy the priors for $\Blambda$ and $\Bnu$ can be chosen to be 
\begin{equation}   \label{eq:priorHyper}
\Blambda \sim \Gam(\la, \lb), \qquad \Bnu \sim \Gam(\nga, \ngb),
\end{equation}
where the parameters $\la, \lb, \nga$ and $\ngb$ are positive constants that, in principle, should be set according to the prior knowledge that is available about the expected values of these parameters. See Appendix \ref{sec:appendix1} for specification of the gamma density and other densities appearing in this paper. 
It is possible to set $\la = 0$ and $\lb = 0$. This `non-informative' prior 
$
p_{\Blambda}(\lambda) \propto \lambda^{-1}
$
can be used if parameters $\la$ and $\lb$ are not to be tuned. Similar non-informative prior can be chosen for $\Bnu$ as well. This way we can, in principle, describe our ignorance about these parameters. Another option could be a 'flat' prior
$
p_{\Blambda}(\lambda) \propto 1.
$
These priors are improper as they do not integrate to $1$. One should note that using an improper prior may or may not lead to an improper posterior and one should always check if this is the case. Nevertheless, we considered the improper priors for the parameters $\Blambda$ and $\Bnu$ in this study and we will discuss later in this paper what kind of issues are related to this choice. We also emphasize at this point that it can be difficult to pick values for these parameters. 

Now with the likelihood~(\ref{eq:Lin}) and the priors (\ref{eq:priorTV}) and (\ref{eq:priorHyper}) the posterior density is, by Bayes' law, 
\begin{align}
  &p_{\Bx,\Bnu,\Blambda \cond \By}( x,\nu,\lambda \cond y ) \nonumber \\
  &\propto \, 
  p_{\Bx,\Bnu,\Blambda}( x,\nu,\lambda) \, p_{\By \cond \Bx,\Bnu,\Blambda}(y \cond x,\nu,\lambda)  \nonumber \\
  &= \,
  p_{\Bx \cond \Blambda}( x \cond \lambda) \, p_{\Blambda}(\lambda) \, p_{\Bnu}(\nu) \, p_{\By \cond \Bx,\Bnu}(y \cond x,\nu)  \nonumber \\
  &\propto \,
  \lambda^{N + \alpha_{\lambda} -1}\nu^{\frac{N}{2} + \alpha_{\nu} -1}
  \e^{ -\left(  \frac{\nu}{2}\|y -H  x \|^2_2
  +\sqrt{\lambda} \sum\limits_{i=1}^{k} \sum\limits_{j=1}^{n} \left\{ \left|\nabla_{ij}^{1}x\right| + \left|\nabla_{ij}^{2}x\right| \right\} + \beta_{\lambda}\lambda + \beta_{\nu}\nu \right)} . 
\end{align}
If the noise and penalisation parameters $\nu$ and $\lambda$ are known, the computation of the MAP estimate can be obtained by computing the minimum of 
\begin{equation}
      \half \|y -H  x \|^2_2  + {\textstyle \frac{\sqrt{\lambda}}{\nu}} 
      \sum_{i=1}^{k} \sum_{j=1}^{n} \left\{ \left|\nabla_{ij}^{1}x\right| + \left|\nabla_{ij}^{2}x\right| \right\}.
\end{equation}
This computation is not trivial due to the absolute values in the TV penalty. The selection of the regularisation parameter $\delta = \frac{\sqrt{\lambda}}{\nu}$ is also a challenge.
However, as shall be shown next, the problem can be made more tractable by introducing additional parameters.


In the discrete TV prior~(\ref{eq:priorTV}), the model coefficients' differences  are conditionally independent Laplace random variables, for instance
\begin{equation} \label{eq:laplacemodel}
    \Bx_{i+1,j}-\Bx_{i,j}\cond (\Blambda = \lambda) \; \stackrel{\text{iid}}{\sim} \; \Laplace(0,\sqrt{\lambda}).
\end{equation}

The multivariate Laplace is a scale mixture of multivariate normal distribution, see Theorem \ref{thm:lgsm} in Appendix \ref{sec:appendix1}. In this case if $\Bx \cond (\Br = r, \Blambda = \lambda) \sim \Normal(0, \frac{2r}{\lambda})$ and $\Br \sim \Exp(1)$ then $\Bx \cond (\Blambda = \lambda) \sim \Laplace(0, \sqrt{\lambda})$.
So, in place of~(\ref{eq:priorTV}), one can use  the prior
\begin{equation}   \label{eq:prior1}
p_{\Bx\cond \Blambda,\Br }(x\cond \lambda,r) \propto 
  \lambda^{N} \prod_{i,j,l}r_{i,j,l}^{-1/2} \,
  \e^{ -\frac{\lambda}{2}\sum\limits_{i=1}^{k} \sum\limits_{j=1}^{n} \left\{ \frac{(\nabla_{ij}^{1}x)^2}{2r_{i,j,1}} + \frac{(\nabla_{ij}^{2}x)^2}{2r_{i,j,2}}\right\}}.
\end{equation}
The difference with model~(\ref{eq:priorTV}) is the addition of  hyperparameters  $\Br_{i,j,l}$, for $i=1,\ldots,k$, $j=1,\ldots,n$, $l=1,2$, where $i$ and $j$ refer to pixels and $l$ either to vertical or horizontal difference of adjacent pixels.
These additional hyperparameters are a-priori independent of $\Blambda$ 
and have exponential prior distributions
\begin{equation}    \label{eq:priorR}
     \Br_{i,j,l} \; \stackrel{\text{iid}}{\sim} \; \Exp(1)
     ,\qquad
     p_{\Br_{i,j,l}}(r_{i,j,l}) = \e^{- {r_{i,j,l}} }, \quad r_{i,j,l} > 0. 
\end{equation}

It is easy to see that
\begin{equation} \label{eq:rdxnorm}
\| R^{-1} D x \|^2_2 = \sum\limits_{i=1}^{k} \sum\limits_{j=1}^{n} \left\{ \frac{(\nabla_{ij}^{1}x)^2}{2r_{i,j,1}} + \frac{(\nabla_{ij}^{2}x)^2}{2r_{i,j,2}}\right\},
\end{equation}
and so the equation (\ref{eq:prior1}) is rewritten in the form
\begin{equation}   \label{eq:prior2}
p_{\Bx\cond \Blambda,\Br }(x\cond \lambda,r) \propto 
  \lambda^{N} \prod_{i,j,l}r_{i,j,l}^{-1/2} \, \e^{-\frac{\lambda}{2}\|  R^{-1} D x \|^2_2 }.
 \end{equation}
In the above formulas $R=\diag(\sqrt{2r}) \in \mathbb{R}^{2N\times 2N}$ with the convention (that we use from now on in this paper) that the square root of a vector is taken component wise. Matrix $D \in \mathbb{R}^{2N\times N}$ above consists of two $N \times N$ blocks corresponding to the differences of components of $x$ in row wise and column wise directions. That is, 
$D=[D_h^T, D_v^T]^T$, where $D_h$ is a $N\times N$ matrix consisting of horizontal differences and similarly 
$N\times N$ matrix $D_v$ contains the vertical differences. This discrete differential operator matrix is assumed throughout this paper. Naturally we could set other kind of matrices $D$. Basically the model allows arbitrary components or their linear combinations to be penalised (as long as the rank condition $\Nul(D) \cap \Nul(H) = \{0\}$, where $\Nul(D)$ denotes the nullspace of matrix $D$, is satisfied). 

The augmented hierarchical model can be summarised by the directed acyclic graph shown in Figure~\ref{fig:dag2}.

Although we are mainly interested in TV and the related Laplace prior, for the sake of generality we will consider more general $\GIG(a,b,p)$ mixing density in the derivations to follow. So we will use
\begin{equation}    \label{eq:priorRgig}
     p_{\Br_{i,j,l}}(r_{i,j,l}) \propto r_{i,j,l}^{p-1} \, \e^{-\half\left(a{r_{i,j,l}} + br_{i,j,l}^{-1}\right) }, \quad r_{i,j,l} > 0.
\end{equation}
in the place of (\ref{eq:priorR}). 
The hierarchical TV prior model is obtainable as a special case because $\GIG(2,0,1) = \Exp(1)$. 


The full posterior for the hierarchical TV model is, by Bayes' law 
\begin{align}
   &p_{\Bx,\Bnu,\Blambda,\Br \cond \By} ( x,\nu,\lambda,r \cond y ) \nonumber \\
  &\propto \,
    p_{\Bx,\Bnu,\Blambda, \Br}(x,\nu,\lambda, r) \,p_{\By \cond \Bx,\Bnu,\Blambda, \Br}(y \cond x,\nu,\lambda, r) \nonumber \\
  &= \,
  p_{\Bx \cond \Blambda, \Br}( x \cond \lambda, r) \, p_{\Blambda}(\lambda) \, p_{\Br}(r) \, p_{\Bnu}(\nu) \, p_{\By \cond \Bx,\Bnu}(y \cond x,\nu) \nonumber \\
  &\propto \,
\lambda^{N +\la -1} \nu^{\frac{N}{2} +\nga - 1}
  \prod_{i,j,l}r_{i,j,l}^{p-3/2} \,
  \e^{-\frac{\nu}{2} \|y-Hx\|^2_2
    -\frac{\lambda}{2}\| R^{-1} D x \|^2_2 - \frac{a}{2}\sum\limits_{i,j,l} r_{i,j,l} - \frac{b}{2}\sum\limits_{i,j,l} r_{i,j,l}^{-1} - \lb\lambda - \ngb\nu} \label{eq:post1} \\
    &=  
  \lambda^{N +\la -1} \nu^{\frac{N}{2} +\nga - 1}
  \prod_{i,j,l}r_{i,j,l}^{p-3/2} \,
  \e^{-\frac{\nu}{2}(x-\hat{x})^TQ(x-\hat{x}) 
  -\frac{\nu}{2}( y^Ty-\hat{x}^TQ\hat{x}) - \frac{a}{2}\sum\limits_{i,j,l} r_{i,j,l} - \frac{b}{2}\sum\limits_{i,j,l} r_{i,j,l}^{-1} - \lb\lambda - \ngb\nu}, \label{eq:post2}
\end{align}
where we have defined $\hat{x} = Q^{-1} H^T y$ and   
$Q = H^TH+{\textstyle\frac{\lambda}{\nu}} D^TR^{-2}D. 
$
It is easy to see that $Q$ is symmetric and positive definite and consequently 
\begin{align} \label{eq:posterior_square_deriv}
 &\|y-Hx\|^2_2 + \frac{\lambda}{\nu}\| R^{-1} D x \|^2_2 \nonumber
 \\ &= (Hx)^T Hx - (Hx)^T y - y^T Hx + y^T y + \frac{\lambda}{\nu}x^TD^TR^{-2}Dx \nonumber
 \\ &=  x^T(H^T H + \frac{\lambda}{\nu}D^TR^{-2}D)x - x^TH^T y - y^T Hx + y^T y \nonumber
 \\ &=  x^TQx - x^TQQ^{-1}H^T y - y^T HQ^{-1}Qx + y^T y \nonumber
 \\ &=  x^TQx - x^TQ\hat{x} - \hat{x}^TQx + \hat{x}^TQ\hat{x} + y^T y - \hat{x}^TQ\hat{x} \nonumber
 \\ &= (x-\hat{x})^TQ(x-\hat{x}) + y^T y - \hat{x}^TQ\hat{x} .
\end{align}
This shows that equations (\ref{eq:post1}) and (\ref{eq:post2}) are indeed equivalent.

\subsection{Conditional densities and MAP estimate}

The conditional distributions are evident by inspection of the full posterior (\ref{eq:post1}) and (\ref{eq:post2}),
\begin{subequations}
\begin{align}
\Bx\cond \nu,\lambda,r, y & \;\sim \;
\Normal(Q^{-1} H^T y,(\nu Q)^{-1}) , 
\\
\Bnu \cond x,\lambda,r, y & \;\sim \; 
      \Gam(\textstyle\frac{N}{2} + \nga,\,  \frac{1}{2}\| y-Hx \|^2_2 + \ngb) ,
\\
\Blambda  \cond x,\nu,r, y & \;\sim \;
      \Gam(N + \la,\,  \textstyle\frac{1}{2} \| R^{-1} D x \|^2_2 + \lb), 
\\
\Br_{i,j,l}  \cond x,\nu,\lambda, r_{-[i,j,l]}, y & \;\sim \;
      \GIG\left(a, \textstyle\frac{1}{2}\lambda (\nabla_{ij}^{l}x)^2 + b, p - \half \right), \nonumber \\
      &\qquad i=1,\ldots,k,\, j=1,\ldots,n, \, l=1,2.
\end{align}
\end{subequations}
The notation $r_{-[i,j,l]}$ means all components of vector $r$ except the $(i,j,l)$th. In the case of anisotropic TV prior (i.e.  $a=2, b=0, p=1$) the last density reverts to 
\begin{equation}
\Br_{i,j,l}  \cond x,\nu,\lambda, r_{-[i,j,l]}, y  \;\sim \;
      \RIG\left(\sqrt{\lambda}|\nabla_{ij}^{l}x|, \textstyle\frac{1}{2}\lambda (\nabla_{ij}^{l}x)^2 \right).
\end{equation}

The MAP estimate for the posterior can be found using IAS (iterative alternating sequential) method that has been used in \cite{Calvetti2008b,Calvetti2008}.
The IAS procedure is to maximize the posterior with respect to each variable (or a group of variables) one at a time keeping all the other variables fixed. One can loop through the variables this way using the newest values of the parameters at each step. The algorithm converges to some local minima since at each step the value of the density cannot decrease. The method is derivative-free. This method is connected to the coordinate descent method in optimisation theory and is presented in e.g.~\cite[Ch. 8.9]{Luenberger2008} and similar method is also known as Lindley-Smith iteration in Bayesian literature \cite{ohagan2004}.

In our case, cycling through the variables can be done via the following equations:
\begin{subequations}
\begin{align}
{x}  & =(H^TH + {\textstyle\frac{{\lambda}}{{\nu}}} D^T{R}^{-2}D)^{-1} H^T y, \label{eq:ias_lin_system}
\\
{\nu} &= \frac{N - 2 + 2\nga}{\| y-H{x} \|^2_2 + 2\ngb} , \label{eq:ias_nu}
\\
{\lambda}  &= 
     \frac{2N - 2 + 2\la}{ \| R^{-1} D x \|^2_2 + 2\lb}, \label{eq:ias_lambda}
\\
{r}_{i,j,l} &=
      \textstyle \frac{p-\frac{3}{2} + \sqrt{\left(p-\frac{3}{2}\right)^2 + a\left(\frac{1}{2}\lambda (\nabla_{ij}^{l}x)^2 + b\right)}}{a}, \nonumber \\ &\quad i=1,\ldots,k,\, j=1,\ldots,n, \, l=1,2.  \label{eq:coordr}
\end{align}
\end{subequations}
These are the formulas for the modes of the conditional densities. In the anisotropic TV case the last equation simplifies to
\begin{equation} \label{eq:coordr_lap}
{r}_{i,j,l} =
      \textstyle -\frac{1}{4} + \frac{1}{4}\sqrt{1 + 4{\lambda}(\nabla_{ij}^{l}x)^2} .
\end{equation}
The resulting IAS prosedure is presented as Algorithm \ref{alg:ias_tv} below.

\begin{algorithm} 
\DontPrintSemicolon
set initial values $r^0, v^0$ and $\lambda^0$ \;
\For{$s=1,2,\ldots$ until stopping criteria is met}{
// use the newest values of the parameters at each step: \;
solve $x^s$ from the linear system (\ref{eq:ias_lin_system}) \; 
compute $\nu^s$ using (\ref{eq:ias_nu}) \;
compute $\lambda^s$ using (\ref{eq:ias_lambda}) \;
\For{all elements in $\{(i,j,l) \,|\, i=1,\ldots,k, \, j=1,\ldots,n, \,  l=1,2 \}$}{
  compute $r_{i,j,l}^s$ from (\ref{eq:coordr}) \; 
}}
return the values $x^\text{end}$, $\nu^\text{end}$, $\lambda^\text{end}$ and $r^{\text{end}}$
\caption{IAS algorithm for TV regularisation in 2d.} \label{alg:ias_tv}
\end{algorithm}

\subsection{Variational Bayes algorithm}

We will derive equations for approximating the posterior mean using VB method. 
The idea of variational inference is to approximate a difficult posterior density to yield useful computational simplifications.
We can sample from the posterior and estimate the conditional mean (CM) using the conditional densities and the resulting Gibbs sampler algorithm. 
However, it is sensible to compute this 'analytic approximation' since sampling based solutions tend to be slow to compute as one typically needs to generate a large number of samples and thus only suitable for small scale problems. Sampling based methods are naturally subject to sampling error and it can be difficult to decide if the MCMC algorithm has converged. 
Compared to MAP estimate, using VB to approximate CM yields also information about the uncertainty of the result and not just point estimates. 


Let $p_{\Bz \cond \By}(z \cond y)$ be the posterior probability density function (pdf) with parameters $z$ and data $y$. Consider pdf $q_{\Bz}$ with partitioned parameters $\Bz = (\Bz_1,\ldots,\Bz_g)$, which can be factorized as
\begin{equation}
 q_{\Bz}(z) = \prod_{i=1}^{g}q_{\Bz_i}(z_i).
\end{equation}
Here the parameters $\Bz_1,\ldots,\Bz_g$ need not be of the same size or one-dimensional. 
The objective is to find such densities $q_{\Bz_i}, i=1,\ldots,g$
that minimize $\KL(q_{\Bz}||p_{\Bz \cond \By})$, the Kullback-Leibler divergence (\cite{KullbackLeibler51}) of density $q_{\Bz}$ with respect to density $p_{\Bz \cond \By}$. This form of variational inference is often called the mean field approximation. 
The optimal pdfs $q_{\Bz_i}^{*}$ can be found by computing the following expectations (see for instance \cite{Fox2012} or \cite[pp. 464 -- 466]{Bishop2006} for derivation)
\begin{equation}
    \ln q_{\Bz_j}^{*}(z_j) = \mean_{\Bz_{i}:i\neq j}[\ln p_{\Bz,\By}(z, y)] + c = \int \ln p_{\Bz,\By}(z, y) \prod_{i\neq j}q_{\Bz_i}(z_i)\,\mathrm{d} z_i + c.
\end{equation}
In the above formula $p_{\Bz,\By}(z, y)$ is the joint distribution of the parameters $\Bz$ and the data $\By$ and $c$ is constant with respect to the current random vector to be solved. The expectation is taken over all variables except the $j$th. The constant is related to the normalisation term and it is unnecessary to be computed since we know that $q_{\Bz_i}^{*}$'s are normalised pdfs. 

The parameters of the distributions $q^{*}_{\Bz_j}$ will usually depend on expectations with respect to other distributions $q^{*}_{\Bz_i}$, $i \neq j$. So the parameters are solved iteratively in practice. That is, one starts with some initial values for the unknown parameters of these pdfs and updates them cyclically using the current estimates for the other densities until some stopping criteria is satisfied.
The algorithm is guaranteed to converge. \cite{Bishop2006}


Next we use VB to approximate the full posterior (\ref{eq:post2}) so that
\begin{equation}
    p_{\Bx,\Bnu,\Blambda,\Br \cond y}(x,\nu,\lambda,r \cond y) \approx
    q_{\Bx} (x)
    q_{\Bnu} (\nu)
    q_{\Blambda} (\lambda)
    q_{\Br} (r).
\end{equation}
%
We will denote the expectation, for example, with respect to a random vectors $\Bnu, \Blambda, \Br$ as $\mean_{\Bnu, \Blambda, \Br}$ when all the other variables are kept fixed. 
A bar over random vector denotes its mean and $c_i$'s denote values that are constants with respect to current variables. These constants are not necessary to be computed. Notice also that matrix $R$ depends on $\Br$. With these conventions it can be calculated that


\begin{tabular}{r c l}
\noalign{\vskip 5mm}   
  $\ln q_{\Bx}^{*}(x)$ &\hspace{-4mm}$=$ &\hspace{-4mm}$\mean_{\Bnu, \Blambda, \Br}\left[\ln p_{\Bx,\Bnu,\Blambda,\Br \cond \By}(x,\Bnu,\Blambda,\Br \cond y)\right] + c_1$  \\
  \noalign{\vskip 1mm}
 &\hspace{-4mm}$=$&\hspace{-4mm}$ \mean_{\Bnu, \Blambda, \Br}\Big[\left(N +\la -1\right)\ln\Blambda + \left(\frac{N}{2} + \nga -1\right)\ln\Bnu   + \left(p-\frac{3}{2}\right)\sum\limits_{i,j,l}^{}\ln \Br_{i,j,l} - \frac{\Bnu}{2}\| y-H x \|^2_2 $ \\  
 \noalign{\vskip 1mm}
 &\hspace{-4mm}&\hspace{-9.5mm} $\quad  - \frac{\Blambda}{2}\| R^{-1} D x \|^2_2  - \frac{a}{2} \sum\limits_{i,j,l}^{} \Br_{i,j,l} - \frac{b}{2} \sum\limits_{i,j,l}^{} \Br_{i,j,l}^{-1} - \lb\Blambda - \ngb\Bnu \Big] + c_1$ \\
 \noalign{\vskip 1mm}
 &\hspace{-4mm}$=$&\hspace{-4mm}$ -\half\bar{\nu}\| y-H x \|^2_2 - \half\bar{\lambda}\mean_{\Br}(\| R^{-1} D x \|^2_2) + c_2$ \\
 \noalign{\vskip 2mm}
 &\hspace{-4mm}$\overset{{(\ref{eq:rdxnorm})}}{=}$&\hspace{-4mm}$ -\half(\bar{\nu}\| y-H x \|^2_2 + \bar{\lambda}\| \bar{R}^{-1} D x \|^2_2) + c_2$ \\
 \noalign{\vskip 1mm}
 &\hspace{-4mm}$\overset{{(\ref{eq:posterior_square_deriv})}}{=}$&\hspace{-4mm}$ -\half(\bar{\nu}(x-\hat{x})^T\bar{Q}(x-\hat{x}) + \bar{\nu}(y^Ty-\hat{x}^T\bar{Q}\hat{x})) + c_2$ \\
 \noalign{\vskip 1mm}
 &\hspace{-4mm}$=$&\hspace{-4mm}$ -\frac{\bar{\nu}}{2}(x-\hat{x})^T\bar{Q}(x-\hat{x}) + c_3,$
\end{tabular}

\noindent
where we have denoted $\hat{x} = \bar{Q}^{-1} H^T y$, $\bar{Q} = H^TH+{\textstyle\frac{\bar{\lambda}}{\bar{\nu}}} D^T\bar{R}^{-2}D$,  
$\bar{R} = \text{diag}\left(\sqrt{2/\mean(\Br^{-1})}\right)$ and thus $\bar{R}^{-2} = \half\text{diag}\left(\mean(\Br^{-1})\right)$ with the convention that the inversions on $\Br$ are done component wise. 
So we obtain
\begin{equation} \label{eq:qx}
\Bx \;\sim\; \Normal(\hat{x}, (\bar{\nu}\bar{Q})^{-1}).
\end{equation}


Next the distribution for $\Bnu$ is derived. We can compute that
\begin{align*}
 \ln q_{\Bnu}^{*}(\nu) &= \mean_{\Bx, \Blambda, \Br}\left[\ln p_{\Bx,\Bnu,\Blambda,\Br \cond \By}(\Bx,\nu,\Blambda,\Br \cond y)\right] + c_1 \\
 &= \left(\frac{N}{2} + \nga -1\right)\ln\nu - \frac{\nu}{2}\mean_{\Bx}(\| y-H \Bx \|^2_2) - \ngb\nu + c_2.
\end{align*}
Using some properties of expectation and trace one can see that  
\begin{align*}
 \mean(\| y-H \Bx \|^2_2) &= \mean(y^Ty-y^TH\Bx-\Bx^TH^Ty + \Bx^TH^TH\Bx) 
 \\ &=  \mean(y^Ty) - \mean(y^TH{\Bx}) - \mean({\Bx}^TH^Ty) + \mean(\tr(\Bx^TH^TH\Bx)) 
 \\ &= y^Ty-y^TH\bar{x} - \bar{x}^TH^Ty  + \mean(\tr(H\Bx\Bx^TH^T)) \\ &= y^Ty-y^TH\bar{x} - \bar{x}^TH^Ty + \tr(H\mean(\Bx\Bx^T)H^T) 
 \\ &= y^Ty-y^TH\bar{x} - \bar{x}^TH^Ty + \tr(H\bar{x}\bar{x}^TH^T) + \tr(H\var(\Bx)H^T) 
 \\ &= y^Ty-y^TH\bar{x} - \bar{x}^TH^Ty + \bar{x}^TH^TH\bar{x} + \tr(\var(\Bx)H^TH) \\ 
 &= \| y-H \bar{x} \|^2_2 + \tr(\var(\Bx) H^TH).
\end{align*}
Thus the distribution for $\Bnu$ is
\begin{equation} \label{eq:qnu}
\Bnu \;\sim\; \Gam\left(\frac{N}{2} + \nga, \half\| y-H \bar{x} \|^2_2 + \half\tr(\var(\Bx) H^TH) + \ngb \right).
\end{equation}


Derivation of the distribution for $\Blambda$ is similar. Since it holds that
\begin{equation}
  \mean_{\Bx,\Br}(\| R^{-1} D \Bx \|^2_2) = \half\sum_{i,j,l}\left(\mean(\Br_{i,j,l}^{-1})\mean((\nabla_{ij}^{l}\Bx)^2)\right),
\end{equation}
which is easily verified using (\ref{eq:rdxnorm}), the result is
\begin{equation} \label{eq:ql}
\Blambda \;\sim\; \Gam\Big(N + \la, \frac{1}{4}\sum_{i,j,l}\left(\mean(\Br_{i,j,l}^{-1})\mean((\nabla_{ij}^{l}\Bx)^2)\right) + \lb \Big). 
\end{equation}
Writing out the square term and using the linearity of expectation shows that 
\begin{equation}
\mean((\nabla_{ij}^{1}\Bx)^2) = \mean((\Bx_{i+1,j}-\Bx_{i,j})^2) = \mean(\Bx_{i+1,j}^2) + \mean(\Bx_{i,j}^2) - 2\mean(\Bx_{i+1,j}\Bx_{i,j}),
\end{equation}
and similarly for horizontal difference when $l=2$. 
Given the mean and covariance matrix of $\Bx$, this statistic can be computed by summing corresponding components of $\mean(\Bx\Bx^T) = \var(\Bx) + \mean(\Bx)\mean(\Bx)^T$. The means for $\Bnu$ and $\Blambda$ can now be computed given the mean and variance of $\Bx$ since the distributions are gamma. For $\Blambda$ and for $\Bx$ one also needs to know the moments $\mean(\Br^{-1})$. 


The components of $\Br$ are mutually independent so $q_{\Br}(r) = \prod_{i,j,l} q_{\Br_{i,j,l}}(r_{i,j,l})$. 
For each of the components $\Br_{i,j,l}$ we obtain a GIG density. The derivation goes as follows.
\begin{align*}
 \ln q_{\Br_{i,j,l}}^{*}(r_{i,j,l}) &= \mean_{\Bx, \Bnu, \Blambda, \Br_{-[i,j,l]}}[\ln p_{\Bx,\Bnu,\Blambda,\Br \cond \By}(\Bx,\Bnu,\Blambda,\Br \cond y)] + c_1 \\
 &= \left(p-\frac{3}{2}\right)\ln r_{i,j,l} -\frac{\bar{\lambda}}{2} \mean_{\Bx,\Br_{-[i,j,l]}}\Big(\sum_{i,j,l}^{}\frac{(\nabla_{ij}^{l}\Bx)^2}{2r_{i,j,l}}\Big) - \half a r_{i,j,l} - \half b r_{i,j,l}^{-1} + c_2 \\
 &= \left(p-\frac{3}{2}\right)\ln r_{i,j,l} -\frac{1}{2r_{i,j,l}}\Big( \frac{\bar{\lambda}}{2}\mean_{\Bx}((\nabla_{ij}^{l}\Bx)^2) + b\Big ) - \half a r_{i,j,l} + c_3 .
\end{align*}
From above it is seen that
\begin{align} \label{eq:qrgig}
\Br_{i,j,l} \;\sim\; \GIG\left(a, \frac{\bar{\lambda}}{2}\mean((\nabla_{ij}^{l}\Bx)^2) + b, p - \half \right), 
\end{align}
for $i=1,\ldots,k, \, j=1,\ldots,n, \,  l=1,2$. 
In the anisotropic TV case, GIG reverts to RIG as in the case of conditional densities and we obtain
\begin{equation} \label{eq:qr}
\Br_{i,j,l} \;\sim\; \RIG\left(\sqrt{\bar{\lambda}\mean((\nabla_{ij}^{l}\Bx)^2)}, \half\bar{\lambda}\mean((\nabla_{ij}^{l}\Bx)^2)\right). 
\end{equation}
Given the parameters, the moments $\mean(\Br_{i,j,l}^{-1})$ in RIG case can be computed from the following formula
\begin{equation} \label{eq:irmoment}
\mean(\Br_{i,j,l}^{-1}) = \frac{2}{\sqrt{\bar{\lambda}\mean((\nabla_{ij}^{l}\Bx)^2)}}.
\end{equation}
A formula for this moment in the case of equation (\ref{eq:qrgig}) can be obtained using result (\ref{eq:gig_moments}) in Appendix \ref{sec:appendix1}.

Starting with some initial values for the unknown parameters of the probability densities above and updating them one at a time using the latest estimates, one finds the optimal distributions for the unknowns $\Bx, \Bnu, \Blambda$ and $\Br$. As a result one can extract the mean of $\Bx$ (which is the same as the mode in this normal distribution case). Furthermore, one can, for example, plot the (marginal) densities of $\Blambda$ and $\Bnu$ to analyze the inference result. The resulting algorithm is presented as Algorithm \ref{alg:vbtv} below. 

\begin{algorithm} 
\DontPrintSemicolon
set initial values of parameters of densities $q^0_{\Br_{i,j,l}}(r_{i,j,l})$, $q^0_{\Bnu}(\nu)$ and $q^0_{\Blambda}(\lambda)$ \;
\For{$s=1,2,\ldots$ until stopping criteria is met}{
// use the newest values of the parameters at each step: \;
solve for parameters of $q^s_{\Bx}(x)$ using (\ref{eq:irmoment}) and (\ref{eq:qx}) \; 
solve for parameters of $q^s_{\Bnu}(\nu)$ using (\ref{eq:qnu}) \;
solve for parameters of $q^s_{\Blambda}(\lambda)$ using (\ref{eq:ql}) \;
\For{all elements in $\{(i,j,l) \,|\, i=1,\ldots,k, \, j=1,\ldots,n, \,  l=1,2 \}$}{
  solve for parameters of $q^s_{\Br_{i,j,l}}(r_{i,j,l})$ using (\ref{eq:qr}) \;
}}
return the parameters of $q^\text{end}_{\Bx}(x)$, $q^\text{end}_{\Bnu}(\nu)$, $q^\text{end}_{\Blambda}(\lambda)$ and $q^{\text{end}}_{\Br_{i,j,l}}(r_{i,j,l})$
\caption{VB algorithm for TV regularisation in 2d.} \label{alg:vbtv}
\end{algorithm}

\subsection{Two-dimensional Laplace TV prior model}

The approach of using Gaussian scale mixtures does not apply for the isotropic TV but we will briefly consider an another related prior. Previously we considered the differences of the neighbouring pixels to be modelled using one-dimensional Laplace densities which was connected to the anisotropic TV. However, next we will briefly discuss the possibility of using bivariate Laplace density. As mentioned earlier, the multivariate Laplace is a scale mixture of the multivariate normal distribution. If $\Sigma = \frac{2r}{\lambda}\Id$, the result can be written so that if $\Bx \cond (\Br = r, \Blambda = \lambda) \sim \Normal(0, \frac{2r}{\lambda}\Id)$ and $\Br \sim \Exp(1)$ then $\Bx \cond (\Blambda = \lambda) \sim \ML(0, \frac{2}{\lambda} \Id)$.
The pdf of this bivariate Laplace distribution is 
\begin{equation}
p_{\Bx\cond\Blambda}(x\cond\lambda) = \frac{\lambda}{2\pi}\Besselk_{0}\left(\sqrt{\lambda}\sqrt{x_1^2+x_2^2}\right) ,
\end{equation} 
where $x = [x_1,x_2]^T \in \mathbb{R}^2$ and function $\Besselk_p$ is the modified Bessel function of the second kind with parameter $p \in \mathbb{R}$. In contrast to the one-dimensional Laplace density, the bivariate Laplace density has a singularity at the origin. 

This connection encourages to study the following two-dimensional TV prior 
\begin{equation}   \label{eq:2dTVwBesselk}
      p_{\Bx\cond \Blambda }(x\cond \lambda)\propto \lambda^{N}
      \prod_{i=1}^k \prod_{j=1}^n \Besselk_0\left(\sqrt{\lambda}\sqrt{(x_{i,j+1}-x_{i,j})^2 + (x_{i+1,j}-x_{i,j})^2}\right).
\end{equation}
The periodic boundary conditions $x_{1,j} = x_{k+1,j}$ and $x_{i,1} = x_{i,n+1}$ are applied as previously. 
Using the same idea as in the case of anisotropic TV, in the place of the prior (\ref{eq:2dTVwBesselk}), that has somewhat difficult formula, one can use the hierarchical prior
\begin{equation}   
p_{\Bx\cond \Blambda,\Br }(x\cond \lambda,r) \propto 
   \lambda^{N} \prod_{i=1}^k \prod_{j=1}^n r_{i,j}^{-1} 
      \,\e^{ -\frac{\lambda}{2}\sum\limits_{i=1}^{k} \sum\limits_{j=1}^{n} \frac{\left(\nabla_{ij}^{1}x\right)^2 + \left(\nabla_{ij}^{2}x\right)^2}{2r_{i,j}} },
\end{equation}
where $\Br_{i,j}$ are additional hyperparameters and $\Br_{i,j} \sim \Exp(1)$. We can again generalise this approach by employing $\GIG(a,b,p)$ mixing density in the place of $\Exp(1)$. The derivation of the posterior, conditional densities and the IAS and VB equations goes in similar fashion as in the case of anisotropic TV. In fact, one can see that the same conditional densities are obtained except that matrix $R$ has changed so that in this case $R = \diag([\sqrt{2r^T}, \sqrt{2r^T}]^T) \in \mathbb{R}^{2N\times 2N}$ and 
\begin{align}
\Br_{i,j}  \cond x,\nu,\lambda, r_{-[i,j]}, y & \;\sim \;
      \GIG\left(a, \textstyle\frac{1}{2}\lambda\big((\nabla_{ij}^{1}x)^2 + (\nabla_{ij}^{2}x)^2\big), p-1 \right) \nonumber
\end{align} 
for $i=1,\ldots,k, \; j=1,\ldots,n$. 
Calculating the mode of this conditional density gives the IAS formula for the hyperparameters $\Br$. The rest of the IAS formulas are the same as in the anisotropic TV case. VB formulas are also easily obtainable so we will not state them here. This prior model could be seen as a sort of hierarchical model of isotropic TV.

\section{Some technical details and remarks} \label{sec:extensions} 

In this section we discuss some technical details and mention some remarks regarding the model. 
There is one issue with the IAS algorithm: if some adjacent pixels become almost the same for some index $(i,j,l)$, then the corresponding latent variable $r_{i,j,l}$ will become very small. 
This issue makes solving the linear system (\ref{eq:ias_lin_system}) for $x$ numerically difficult as infinite values must be handled. One can use the matrix inversion lemma to deal with the issue. However, for this to work, it is required that $T=D^T{R}^{-2}D$ is invertible which is clearly true in the Lasso case since $D=\Id$ and thus $T$ is diagonal. See e.g.~\cite{Figueiredo2003} for some details. In general case one can replace the exponential hyperprior with some approximation to it like $\GIG(2,0.001,1)$ which will ensure that the values for $r$ cannot go exactly to zero in the equation (\ref{eq:coordr}).
Of course, then the prior is no more exactly related to the Laplace prior but to some other heavy-tailed approximation to it that has less sharp peak at the origin. This approximation is similar to the approximation $|x| \approx \sqrt{x^2 + \beta^2}$ for small positive $\beta$ (see, e.g. \cite{Vogel2002}) which is often used to make the TV functional differentiable at the origin. This approximation allows one to use gradient-based optimisation methods to tackle the problem. Alternatively, it might be possible to make the latent variables to be thresholded to keep them above some small positive value but we did not consider this idea.

Setting $\GIG(0,w,-\frac{w}{2}) = \InvGam(\frac{w}{2},\frac{w}{2})$ as the mixing density leads to a hierarchical Student's t-distribution prior model. In this model one needs to specify the degree of freedom $w$ for the prior. Since $a=0$ we see that the resulting densities for hyperparameters are also inverse gamma densities with a mode that does not present singularity issues. With small values of $w$ one could expect stronger penalisation for smoothness. With large values of $w$ the results will be smoother as the t-distribution converges to Gaussian as $w \rightarrow \infty$.

Similar derivations as in the case of the anisotropic TV can easily be done in one-dimensional setting. 
In addition, as a special case, replacing matrix $D$ with identity matrix and the differences in the sums that are written out with single components, gives a hierarchical Bayesian model for the Lasso. Similar hierarchical Lasso models have been considered e.g. in \cite{Figueiredo2003, Park2008, Kyung2010} mainly for regression problems.


Inverting the covariance matrix of $\Bx$ is required in the VB algorithm. This limits the use of VB only to problems with small dimension. Perhaps the matrix inversions could be done faster by assuming bccb (block circulant with circulant blocks) structure for covariance matrix and inverting it in Fourier domain as used in \cite{Babacan2007} and using certain approximations. However, even if that could be somehow done (without large approximation errors), this method does not work for general matrix. In \cite{Calvetti2008b} Region of Interest (ROI) method was introduced for one to be able to compute the CM estimate and assess uncertainty of the result in a small (intersting!) domain of the whole image. We did not consider this approach in our case but it might be possible to use similar method in out VB case.

In the case of the MAP estimate only a linear system has to be computed which is evidently faster and allows computing the MAP estimate in much larger dimensions. It might be possible to do it in Fourier domain but, even if the matrix does not have any nice structure to exploit, iterative techniques such as the conjugate gradient method with preconditioning can be used \cite[Ch. 5]{Vogel2002} to speedup the computation. A drawback compared to Tikhonov regularisation formula is that one needs to solve this linear system as many times as iteration steps are needed for convergence. We also note that the iterative methods for solving the linear system developed in \cite{Calvetti2008b} (and in some of the preceding papers by the authors of \cite{Calvetti2008b}) could be used in our case. 

The similarity between the conditional densities, IAS equations and the independent densities of VB method cannot be ignored. For instance, in VB case the means and certain moments with respect to other independent densities appear while in the corresponding conditional densities of the Gibbs sampler one uses samples drawn from the other corresponding conditional densities. Futhermore, in IAS algorithm the modes of the conditional densities are used. The expectation maximisation (EM) method has a variant called expectation conditional maximisation (ECM) \cite{Ecm1993} which could also have been used and would have produced quite similar formulas as the IAS (or VB) method. In EM based approaches usually the hyperparameters are considered as latent variables. In ECM the maximisation step is done in the style of IAS by maximising the log-likelihood with respect to each variable. Unlike in standard EM it is enough that the value of the objective function increases but the exact maximum is not needed to be solved. EM algorithm has been applied for the Lasso case in \cite{Figueiredo2003}.

\section{Examples of image deblurring} \label{sec:examples}

In this section we demonstrate the algorithms presented in earlier sections with one and two-dimensional image deblurring test problems. There are several priors (obtainable by setting different GIG mixing densities) that can be used in our model for the following test problems. However, for the sake of brevity we only show results for priors that are hierarchical representations of Laplace and t-distributions. We will not discuss the differences between the priors in detail here but focus more on the general aspects. 

Non-informative priors for $\Blambda$ and $\Bnu$ were used in the test problems although it could be possible to set them according to prior knowledge e.g.~about the noise level in the images. In the IAS algorithm in the case of Laplace type of prior we used $\GIG(2,0.001,1)$ mixing density to avoid the singularity issues as discussed in Section \ref{sec:extensions}. The degree of freedom in the case of the t-distribution priors was set to 2 in all of our examples. All the test images were convolved with Gaussian blurring kernel and then white Gaussian noise was added to the blurred test images. The results were computed with \textsc{Matlab} 2011b. 


We start with simple one-dimensional examples. 
Some reconstructions using an implementation based on our hierarchical model are presented in Figures \ref{fig:blur1d_demo1b} and \ref{fig:blur1d_demo2}. 
In these examples the signal was defined on the interval $[0,1]$ and it was divided into equispaced $100$ points. For comparison, we used Tikhonov regularisation with penalty term $J(x)=\delta\|Dx\|_2^2$ and a deterministic TV regularisation algorithm that is based on the well-known method of transforming the problem into a quadratic programming problem that is easier to solve. 


In two-dimensional image deblurring examples we used a small $42 \times 42$ pixel test image and a moderate size $200 \times 200$ Shepp-Logan phantom image which were blurred using $7 \times 7$ Gaussian blur mask. Some reconstructions of these images are shown in Figures \ref{fig:blur2d_demo} and \ref{fig:blur2d_phantom_demo}.

In these test cases the introduced hierarchical model provides promising estimation results. To our eye the reconstruction qualities are mostly equal or only slightly worse than those obtained by the deterministic optimisation algorithm whose regularisation parameters were tuned 'by eye' for the best reconstrution performance. 
We find that there is little difference between CM estimates computed by Gibbs sampling and VB. This indicates that the VB approximation is accurate. We note, however, that even though the  mean values seem to agree quite well, it can be that the densities are still different. In fact, VB seems to give slightly better image reconstructions than the Gibbs sampler algorithm although neither of them is not well suited for reconstructing sharp edges of the images. The Gibbs sampling was performed using the conditional densities that we derived in Section \ref{sec:model}.
In the Gibbs sampler $10000$ samples were generated from the posterior. 
The CM estimates were not much better than Tikhonov regularisation, though in our examples with less noise or smaller blur levels it performed better (not shown). 

The MAP estimate worked better for blocky images due to its 'sparsity' property than the CM estimate which tends to produce smoother solutions and does not preserve the edges very well with higher blurring or noise levels. The MAP estimate, on the other hand, produced a 'staircasing effect' with the smooth curve in the first example indicating that the regularisation parameter was estimated to be perhaps too large for this case. 
The t-distribution TV prior models produced results with surprisingly well preserved edges although the density is not sharp peaked like Laplace density. With larger values the reconstructions tended to become smoother because then the prior approaches Gaussian prior. With small degree of freedom levels t-distribution prior favoured indeed very 'sparse' solutions and performed thus well in the case of blocky image. 

The IAS and VB algorithms converged reasonably fast in these tests. Usually less than 50 iterations were needed for the convergence. The needed iterations (and the reconstruction performance) expectedly depend on how small changes in the estimated image were used as a stopping criteria for the algorithms. 

In small dimensional cases we solved the linear systems in our equations directly but in large dimensional cases (e.g.~the $200 \times 200$ Shepp-Logan phantom image) we had to use iterative strategy i.e.~preconditioned conjugate gradient method to speedup the computations. It was also noticed that if the initial values were set very badly or if extreme amount of noise was added to the images then the algorithms converged to unwanted solutions like a blank image (as the regularisation parameter converged close to infinity) or to the original blurred and noisy image (as the regularisation parameter converged close to zero). However, in most cases there were no problems with the convergence of the algorithms in the case of several deblurring test examples. 

We also tested the models with certain image denoising test problems. This was done by simply setting the matrix $H$ to be the identity matrix. The results of these problems were mostly not good and issues with the convergence of the algorithms were encountered. The estimates often converged to either blank image or to the original noisy image. The results of these examples are not illustrated here. The convergence issues might be because the posterior could be improper and the computation using the IAS algorithm fails. If this is the case then the results by the Gibbs sampling or VB might also be somewhat questionable. Setting tight priors for $\lambda$ and $\nu$ parameters allows convergence to expected results (and ensures properness of the posterior) but then one faces the problem of choosing these values which is exactly what we wanted to avoid by studying the hierarchical model in this paper. It is also likely the case that the posterior is dominated by the prior of the parameters $\lambda$ and $\nu$. This makes choosing the prior very crucial but generally one does not have enough prior information (similarly as one does not know how large the regularisation parameter should be set in the deterministic framework). Also, generally, if the posterior is multimodal, then it is possible that the IAS algorithm converges to some local maxima. 
Thus it seems that our method is partly unsatisfactory. 
However, in the case of image deblurring, the improper priors were good choice and in most test cases, the results were quite good and no tuning were needed. Anyway, we hope to conduct more tests on the convergence properties of the algorithms in near future.

\section{Conclusions} \label{sec:concl}

In this work total variation regularisation was studied in the Bayesian context. The usual deterministic optimisation algorithms require auxiliary methods for determining regularisation parameter and produce only point estimates. In the model proposed here, the total variation penalty function was formulated as a Laplace prior distribution and the posterior for the model was derived exploiting the Gaussian scale mixture property. The formulation is straightforward to generalise to other heavy-tailed `TV like' priors that promote sparsity of the estimated images and all the essential parameters in the model are simultaneously estimated from the data. The uncertainty of the results can be (at least in principle) assessed. 
Algorithms for the conditional mean and maximum a posteriori estimates were derived using the variational Bayes and IAS methods. 

When compared to other more or less similar approaches in the literature, our model is more flexible and slightly more general. Although we did not test the differences between the different TV priors exhaustively it seems that the Laplace TV prior and the t-distribution TV prior with a small degree of freedom are good choices in cases where preserving the edges of the image is essential. We also proposed the two-dimensional Laplace TV prior that can be considered a statistical alternative for isotropic TV prior. 

Generally the MAP estimates worked well in our deblurring test cases for restoring blocky images and the edges were well preserved. 
Comparison of the MAP and (approximative) CM estimates showed that the CM tends to yield more smooth and less 'edge preserving' image reconstructions than the MAP estimate. 
The CM estimates computed using VB were very close to those results obtained using Gibbs sampling. Unfortunately computing the (approximate) CM estimates is computationally intensive because of the need to invert large matrices in the VB algorithm. In \cite{Calvetti2008b} Region of interest (ROI) method was developed because of the similar problems. It might be possible to apply the same approach in our VB case also. The MAP estimate, on the other hand, requires only solving a linear system for which special iterative techniques can be used. 

Our model in which all the parameters are estimated from data is not fully satisfactory. Although realistic estimates could be obtained in the case of deblurring problems, difficulties with the estimation were encountered in some cases as the algorithms did not converge to proper solutions. This might be because the posterior can become improper if using improper priors or even if is not, it is likely the case that the prior for the hyperparameters dominates the likelihood. One can, of course, set tight priors for the parameters or simply consider them to be fixed (which is the same as choosing Dirac delta priors) but the need to do this is against the motivation of this study. If the $\nu$ and $\lambda$ parameters are fixed, then our model simplifies to a model considered in \cite{Calvetti2008b} (although there are some other differences). It might be possible to analyse these properties of the model both analytically and numerically in more detailed way in future. 

As future work also more comprehensive study of hierarchical models could be carried out and one could implement and test the algorithms with more specific imaging (or other) problems. 
Although the proposed method does not require separate tuning of the regularisation parameter in the deblurring case in typical scenarios, the method does include higher-level parameters (the GIG distribution's parameters) that could be tuned to give good results for a particular problem class.
Finally, in the model we assumed Gaussian noise. However, by using Gaussian scale mixture trick it is also possible to generalise the method so that heavy-tailed noise of the image is modeled.

\bibliography{article_arxiv}

\begin{thebibliography}{10}

\bibitem{Andrews1974}
D.~F. Andrews and C.~L. Mallows.
\newblock Scale mixtures of normal distributions.
\newblock {\em Journal of the Royal Statistical Society. Series B
  (Methodological)}, 36(1):99--102, 1974.

\bibitem{Babacan2007}
S.~D. Babacan, R.~Molina, and A.~K. Katsaggelos.
\newblock {Total Variation Image Restoration and Parameter Estimation Using
  Variational Posterior Distribution Approximation.}
\newblock In {\em {International Conference on Image Processing (1)}}, pages
  97--100. IEEE, 2007.

\bibitem{Babacan09}
S.~D. Babacan, R.~Molina, and A.~K. Katsaggelos.
\newblock {Variational Bayesian Blind Deconvolution Using a Total Variation
  Prior.}
\newblock {\em IEEE Transactions on Image Processing}, 18(1):12--26, 2009.

\bibitem{Babacan2010}
S.~D. Babacan, R.~Molina, and A.~K. Katsaggelos.
\newblock {Bayesian compressive sensing using Laplace priors}.
\newblock {\em IEEE Transactions on Image Processing}, 19(1):53--63, January
  2010.

\bibitem{Bishop2006}
C.~M. Bishop.
\newblock {\em Pattern Recognition and Machine Learning}.
\newblock Springer, New York, 2006.

\bibitem{Calvetti2010}
D.~Calvetti, H.~Hakula, S.~Pursiainen, and E.~Somersalo.
\newblock {Conditionally Gaussian Hypermodels for Cerebral Source
  Localization.}
\newblock {\em SIAM Journal of Imaging Sciences}, 2(3):879--909, 2009.

\bibitem{Calvetti2008b}
D.~Calvetti and E.~Somersalo.
\newblock {Hypermodels in the Bayesian imaging framework}.
\newblock {\em Inverse Problems}, 24(3):034013, 2008.

\bibitem{Calvetti2008}
D.~Calvetti and E.~Somersalo.
\newblock {Recovery of Shapes: Hypermodels and Bayesian Learning}.
\newblock In {\em {Journal of Physics: Conference Series}}, volume 124, 2008.

\bibitem{Chantas2010}
G.~K. Chantas, N.~P. Galatsanos, R.~Molina, and A.~K. Katsaggelos.
\newblock {Variational Bayesian Image Restoration With a Product of Spatially
  Weighted Total Variation Image Priors.}
\newblock {\em IEEE Transactions on Image Processing}, 19(2):351--362, 2010.

\bibitem{Eltoft2006}
T.~Eltoft, T.~Kim, and T.~Lee.
\newblock On the multivariate {Laplace} distribution.
\newblock {\em IEEE Signal Processing Letters}, 13(5), June 2006.

\bibitem{Figueiredo2003}
M.~A.~T. Figueiredo.
\newblock {Adaptive Sparseness for Supervised Learning}.
\newblock {\em IEEE Transactions on Pattern Analysis and Machine Intelligence},
  25(9):1150--1159, September 2003.

\bibitem{Fox2012}
C.~Fox and S.~Roberts.
\newblock {A Tutorial on Variational {Bayesian} Inference}.
\newblock {\em Artificial Intelligence Review}, 38(2):85--95, 2012.

\bibitem{GoldsteinO09}
T.~Goldstein and S.~Osher.
\newblock {The Split Bregman Method for L1-Regularized Problems.}
\newblock {\em SIAM Journal on Imaging Sciences}, 2(2):323--343, 2009.

\bibitem{Jin2010}
B.~Jin and J.~Zou.
\newblock Hierarchical {B}ayesian inference for ill-posed problems via
  variational method.
\newblock {\em Journal of Computational Physics}, 229(19):7317--7343, September
  2010.

\bibitem{Jorgensen1982}
B.~Jorgensen.
\newblock {\em Statistical Properties of the Generalized Inverse {Gaussian}
  Distribution}, volume~9 of {\em Lecture Notes in Statistics}.
\newblock Springer-Verlag, New York, 1982.

\bibitem{Kaban07}
A.~Kab\'an.
\newblock {On Bayesian classification with Laplace priors.}
\newblock {\em Pattern Recognition Letters}, 28(10):1271--1282, 2007.

\bibitem{Kaipio2004}
J.~P. Kaipio and E.~Somersalo.
\newblock {\em Computational and Statistical Methods for Inverse Problems}.
\newblock Springer, 2004.

\bibitem{Kotz2001}
S.~Kotz, T.~J. Kozubowski, and K.~Podg{\'o}rski.
\newblock {\em The {L}aplace Distribution and Generalizations: A Revisit with
  Applications to Communications, Economics, Engineering, and Finance}.
\newblock Progress in Mathematics Series. Birkh{\"a}user, 2001.

\bibitem{KullbackLeibler51}
S.~Kullback and R.~A. Leibler.
\newblock {On Information and Sufficiency}.
\newblock {\em The Annals of Mathematical Statistics}, 22:79--86, 1951.

\bibitem{Kyung2010}
M.~Kyung, J.~Gill, M.~Ghosh, and G.~Casella.
\newblock Penalized regression, standard errors, and {Bayesian} lassos.
\newblock {\em Bayesian Analysis}, 5(2):369--412, 2010.

\bibitem{Lucka11}
F.~Lucka.
\newblock Hierarchical {Bayesian} approaches to the inverse problem of
  {EEG/MEG} current density reconstruction.
\newblock German diploma thesis (mathematics), Institute for Computational and
  Applied Mathematics, University of Muenster, March 2011.

\bibitem{Lucka12}
F.~Lucka.
\newblock {Fast Markov chain {Monte Carlo} sampling for sparse {Bayesian}
  inference in high-dimensional inverse problems using L1-type priors}.
\newblock {\em Inverse Problems}, 28(12):125012, September 2012.

\bibitem{Luenberger2008}
D.~G. Luenberger and Y.~Ye.
\newblock {\em Linear and Nonlinear Programming}.
\newblock Springer, New York, third edition, 2008.

\bibitem{MacKay2003}
D.~J.~C. MacKay.
\newblock {\em Information Theory, Inference, and Learning Algorithms}.
\newblock Cambridge University Press, 2003.

\bibitem{Ecm1993}
X.~Meng and D.~B. Rubin.
\newblock Maximum likelihood estimation via the {ECM} algorithm: A general
  framework.
\newblock {\em Biometrika}, 80(2):267--278, June 1993.

\bibitem{Nummenmaa2007}
A.~Nummenmaa, T.~Auranen, M.~S. Hämäläinen, I.~P. Jääskeläinen,
  J.~Lampinen, M.~Sams, and A.~Vehtari.
\newblock {Hierarchical Bayesian estimates of distributed MEG sources:
  Theoretical aspects and comparison of variational and MCMC methods}.
\newblock {\em NeuroImage}, (3):947--966, 2007.

\bibitem{ohagan2004}
A.~O'Hagan and J.~Forster.
\newblock {\em {Advanced Theory of Statistics, Bayesian inference}}.
\newblock Arnold London, 2004.

\bibitem{Park2008}
T.~Park and G.~Casella.
\newblock {The Bayesian Lasso}.
\newblock {\em {Journal of the American Statistical Association}}, 103(482),
  June 2008.

\bibitem{Rudin1992}
L.~I. Rudin, S.~Osher, and E.~Fatemi.
\newblock {Nonlinear total variation based noise removal algorithms}.
\newblock {\em Physica D: Nonlinear Phenomena}, 60(1-4):259--268, 1992.

\bibitem{Tweedie1956}
M.~C.~K. Tweedie.
\newblock Statistical properties of inverse {Gaussian} distributions. {I}.
\newblock {\em {Annals of Mathematical Statistics}}, 28(2):362--377, 1956.

\bibitem{Vogel2002}
C.~R. Vogel.
\newblock {\em Computational Methods for Inverse Problems}.
\newblock Number~10 in Frontiers in Applied Mathematics. Society for Industrial
  and Applied Mathematics (SIAM), Philadelphia, 2002.

\bibitem{Wipf2009}
D.~P. Wipf and S.~S. Nagarajan.
\newblock {A unified Bayesian framework for MEG/EEG source imaging.}
\newblock {\em NeuroImage}, 44(35):947--966, 2009.

\bibitem{Zuo2011}
W.~Zuo and Z.~Lin.
\newblock {A Generalized Accelerated Proximal Gradient Approach for
  Total-Variation-Based Image Restoration.}
\newblock {\em IEEE Transactions on Image Processing}, 20(10):2748--2759, 2011.

\end{thebibliography}
\bibliographystyle{plain}

\appendix
\section{Probability distributions} \label{sec:appendix1}

The probability densities used in this paper and related facts are summarised here. 
We start by the generalised inverse Gaussian (GIG) distribution \cite{Jorgensen1982}. 
\begin{definition}\label{gigdef}
A random variable $\Bx > 0$ has GIG distribution, denoted $\Bx \sim \GIG(a,b,p)$, with parameters $a, b$ and $p$ if it has the pdf
\begin{align}
  p_{\Bx }(x) = \frac{(a/b)^{p/2}}{2K_p(\sqrt{ab})} x^{p-1}\e^{-\frac{1}{2}(ax+\frac{b}{x})}, \quad x > 0,
\end{align}
where $K_p$ is the modified Bessel function of the second kind with parameter $p \in \mathbb{R}$. The range of the parameters is
\begin{align}\label{eq:k_defarea}
 a > 0, \; b \geq 0, \; p > 0; \quad a > 0, \; b > 0, \; p = 0; \quad a \geq 0, \; b > 0, \; p < 0.
\end{align}
\end{definition}

The GIG distribution is unimodal and skewed. The central moments, mode and variance for GIG can be computed using formulas \cite[pp. 7, 13--14]{Jorgensen1982} 
%
%
\begin{align}
\mean (\Bx^q) &= \left(\frac{b}{a}\right)^{q/2} \frac{\Besselk_{p+q}(\sqrt{ab})}{\Besselk_{p}(\sqrt{ab})}, \quad q \in \mathbb{R}, \label{eq:gig_moments} \\
\mode(\Bx) &= 
\begin{cases}
\frac{(p-1) + \sqrt{(p-1)^2 + ab}}{a}, & \text{if } a > 0, \\ 
\frac{b}{2(1-p)}, & \text{if } a = 0 ,
\end{cases} \\
\var (\Bx) &= \frac{b}{a} \left( \frac{\Besselk_{p+2}(\sqrt{ab})}{\Besselk_{p}(\sqrt{ab})} - 
   \left(\frac{\Besselk_{p+1}(\sqrt{ab})}{\Besselk_{p}(\sqrt{ab})}\right)^{2} \right), \label{eq:gig_var}
\end{align}
which can be further simplified if $a=0$ or $b=0$ using certain asymptotic property of the modified Bessel function. These formulas are given in \cite[pp. 13--14]{Jorgensen1982}. 

Reciprocal Inverse Gaussian (RIG) is a special case of GIG. Setting 
$a=\alpha^2/\beta, b=\beta$ and $p=1/2$ gives RIG. Also gamma and inverse gamma densities are special cases of GIG which follow by setting $a=2\beta, b=0$ and $p=\alpha$ or $a=0, b=2\beta$ and $p=-\alpha$, respectively. Exponential distribution $\Exp(\theta)$ is the same as $\Gam(1,\theta)$. These densities and some of their statistics  are gathered in Tables \ref{table:gig} and \ref{table:gigg}. Note that $\Gamma(\cdot)$ is the gamma function and is defined as $\Gamma(x) = \int_0^{\infty}t^{x-1}\e^{-t} \mathrm{d} t$ for $x>0$. RIG density has also been studied in \cite{Tweedie1956}.



We define the multivariate Laplace distribution in the following way.
\begin{definition}
A random $n$-vector $\Bx$ is said to have a multivariate Laplace distribution, denoted as $\Bx \sim \ML(\mu,\Sigma)$, with parameters $\mu$ and a $n\times n$ positive definite matrix $\Sigma$, if it has the pdf
\begin{align}
   p_{\Bx}(x) = \frac{2}{(2\pi)^{n/2} (\det(\Sigma))^{1/2}} \frac{\Besselk_{\frac{n}{2}-1}\left(\sqrt{2(x-\mu)^T\Sigma^{-1}(x-\mu)}\right)}{\left(\sqrt{\half(x-\mu)^T\Sigma^{-1}(x-\mu)}\right)^{\frac{n}{2}-1}}.
\end{align}
\end{definition}
This definition agrees with the one in \cite[p. 235]{Kotz2001} but with the added location parameter $\mu$. 
%
In one dimension the Laplace pdf with $\Sigma = 2/b^2 \in \mathbb{R}_{+}$ and $b>0$ reduces to 
\begin{align}
 p_{\Bx}(x) = \frac{b}{2}\e^{-b |x-\mu|}.
\end{align}
This one-dimensional Laplace density is denoted as $\Laplace(\mu, b)$. This is easily seen by some simple calculations and using a property of the Bessel function. The second parameter was chosen in this specific way for convenience. 
The density is sometimes also called the double-exponential density as it consists of two exponential curves.

The mean of this density is $\mu$ and variance is $\Sigma$, which can be seen using the following connection to Gaussian and exponential density. 
The multivariate Laplace distribution defined above is a Gaussian scale mixture with exponential mixing density. That is, it can be written as 
\begin{align}\label{ml_gsm}
\By = \mu + \sqrt{\Br}\Sigma^{1/2}\Bx,
\end{align} 
where $\Br \sim \Exp(1)$, $\Bx \sim \Normal(0,\Id)$ and $\Br$ and $\Bx$ are independent. The square root is defined as being any matrix such that $\Sigma = (\Sigma^{1/2})(\Sigma^{1/2})^T$. The fact that Laplace density can be represented as GSM in one-dimensional case was discovered by Andrews and Mallows in 1974 \cite{Andrews1974}. 
The next theorem is the generalisation to multidimensional case inspired by \cite{Eltoft2006}, where a slightly different version of the result below is given.


\begin{theorem}[Laplace as Gaussian scale mixture] \label{thm:lgsm}
If $\By \cond (\Br=r) \sim \Normal(\mu, r \Sigma)$ and $\Br \sim \Exp(1)$, then $\By \sim \ML(\mu, \Sigma)$.
\end{theorem}
\begin{proof}
The proof is straightforward calculation. 
\begin{align*}
p_{\By}(y) &= \int_0^{\infty}p_{\By,\Br}(y,r) \mathrm{d} r = \int_0^{\infty}p_{\By\cond \Br}(y\cond r) p_{\Br}(r) \mathrm{d} r \\
&= \int_0^{\infty} \frac{1}{(2\pi)^{n/2}(\det(r \Sigma))^{1/2}} \e^{-\frac{1}{2r}z^T\Sigma^{-1}z} \e^{-r} \mathrm{d} r \\
&= \frac{1}{(2\pi)^{n/2}(\det(\Sigma))^{1/2}} \underbrace{\int_0^{\infty} r^{-n/2} \e^{-\half\left(2r + \frac{z^T\Sigma^{-1}z}{r}\right)} \mathrm{d} r}_{ = 2\Besselk_{1-\frac{n}{2}}\left(\sqrt{2z^T\Sigma^{-1}z}\right) (\half z^T\Sigma^{-1}z)^{\half\left(1-\frac{n}{2}\right)}} \\
&= \frac{2}{(2\pi)^{n/2} (\det(\Sigma))^{1/2}} \frac{\Besselk_{\frac{n}{2}-1}\left(\sqrt{2z^T\Sigma^{-1}z}\right)}{\left(\sqrt{\half z^T\Sigma^{-1}z}\right)^{\frac{n}{2}-1}},
\end{align*}
where we have denoted $z=y-\mu$. The integrand on the third line is recognised as unnormalised $\GIG(2, z^T\Sigma^{-1}z, 1-\frac{n}{2})$ pdf and is thus computed using the fact that the density integrates to $1$.
\end{proof}

Similarly it can be shown that the multivariate t-distribution can be characterised as a GSM using the connection $\By = \mu + \sqrt{\Br}\Sigma^{1/2}\Bx$ with inverse gamma mixing density $\Br \sim \InvGam(\frac{w}{2}, \frac{w}{2})$, where $w$ is the degree of freedom for t-distribution.

\newpage

\begin{table}
\caption{Special cases of GIG distribution. All the parameters appearing in the formulas must be positive.} 
\begin{center}
 {\begin{tabular}{c c}
 \hline
 \noalign{\vskip .6mm}
 $\Bx \sim$  & $p_{\Bx}(x)$  \\ \hline
 \noalign{\vskip 1mm}    
 \noalign{\vskip 1mm} 
 $\RIG(\alpha,\beta)$ & $\frac{\alpha}{\sqrt{2\pi\beta}} \e^{2\alpha} x^{-\frac{1}{2}} \exp\left({-\frac{(\alpha x+\beta)^2}{2\beta x} }\right)$ \\
 $\Exp(\theta)$ & $\theta \e^{-\theta x}$ \\
 \noalign{\vskip 1mm}
 $\Gam(\alpha,\beta)$ & $\frac{\beta^{\alpha}}{\Gamma(\alpha)}x^{\alpha-1} \e^{-\beta x}$ \\ 
 \noalign{\vskip 1mm}
 $\InvGam(\alpha,\beta)$ & $\frac{\beta^{\alpha}}{\Gamma(\alpha)}x^{-\alpha-1} \e^{-\beta /x}$ \\ \noalign{\vskip 1mm} \hline
 \end{tabular}}
\end{center}
\label{table:gig} 
\end{table}

\newpage

\begin{table}
\caption{Some statistics of different distributions.} 
\begin{center}
 {\begin{tabular}{c c c c}
 \hline
 \noalign{\vskip .6mm}
 $\Bx \sim$ & $\mean(\Bx)$ & $\mode(\Bx)$ & $\var(\Bx)$  \\ \hline
 $\RIG(\alpha,\beta)$ & $\frac{\beta(1+\alpha)}{\alpha^2}$ &  $\frac{-\beta + \beta\sqrt{1+4\alpha^2}}{2\alpha^2}$ & use (\ref{eq:gig_var}) \\
 \noalign{\vskip 1mm}
 $\Exp(\theta)$ & $\theta^{-1}$ & $0$ & $\theta^{-2}$ \\ 
 \noalign{\vskip 1mm}
 $\Gam(\alpha,\beta)$ & $\frac{\alpha}{\beta}$ & $\frac{\alpha-1}{\beta}$ for $\alpha > 1$ & $\frac{\alpha}{\beta^{2}}$ \\
 \noalign{\vskip 1mm}
 $\InvGam(\alpha,\beta)$ & $\frac{\beta}{\alpha-1}$ for $\alpha > 1$ & $\frac{\beta}{\alpha+1}$, & $\frac{\beta^2}{(\alpha-1)^{2}(\alpha-2)}$ for $\alpha > 2$ \\ \noalign{\vskip 1mm} \hline
 \end{tabular}}
\end{center}
\label{table:gigg} 
\end{table}

\newpage

\begin{figure}[t]
\begin{center}
\resizebox{.32\textwidth}{!}{\includegraphics{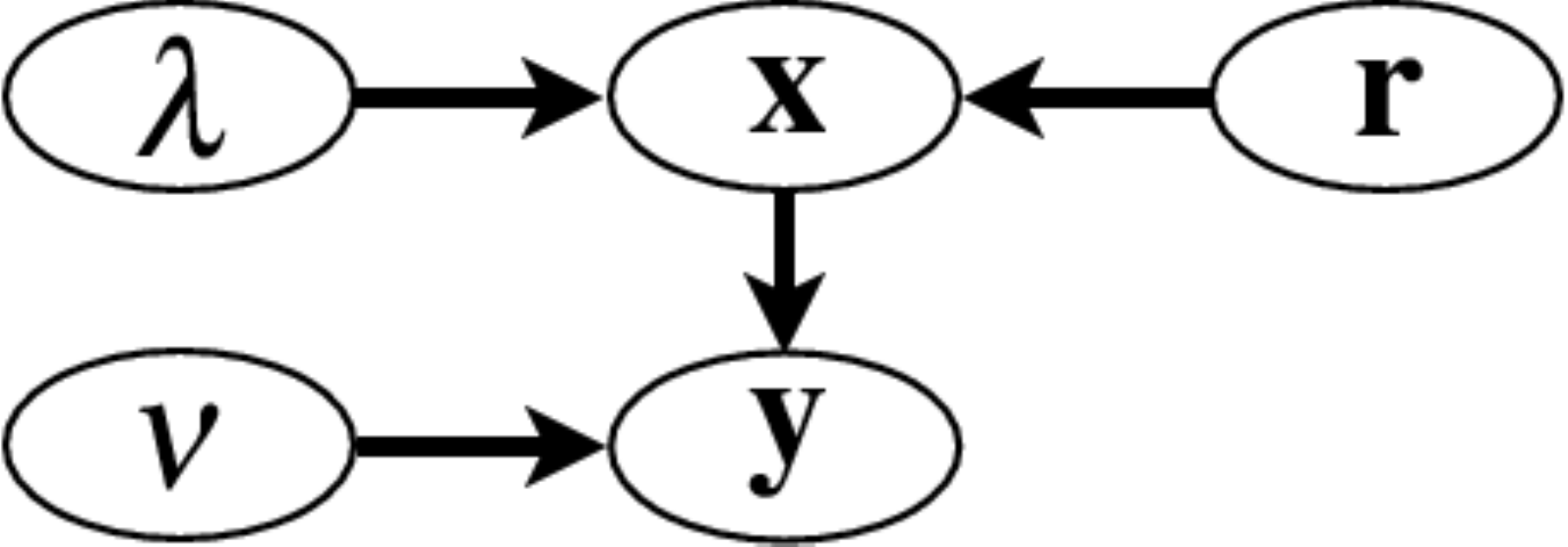}}
\caption{Graphical model of TV regularisation.
\label{fig:dag2}}
\end{center}
\end{figure}

\begin{figure}
\begin{center}
\resizebox{.6\textwidth}{!}{\includegraphics{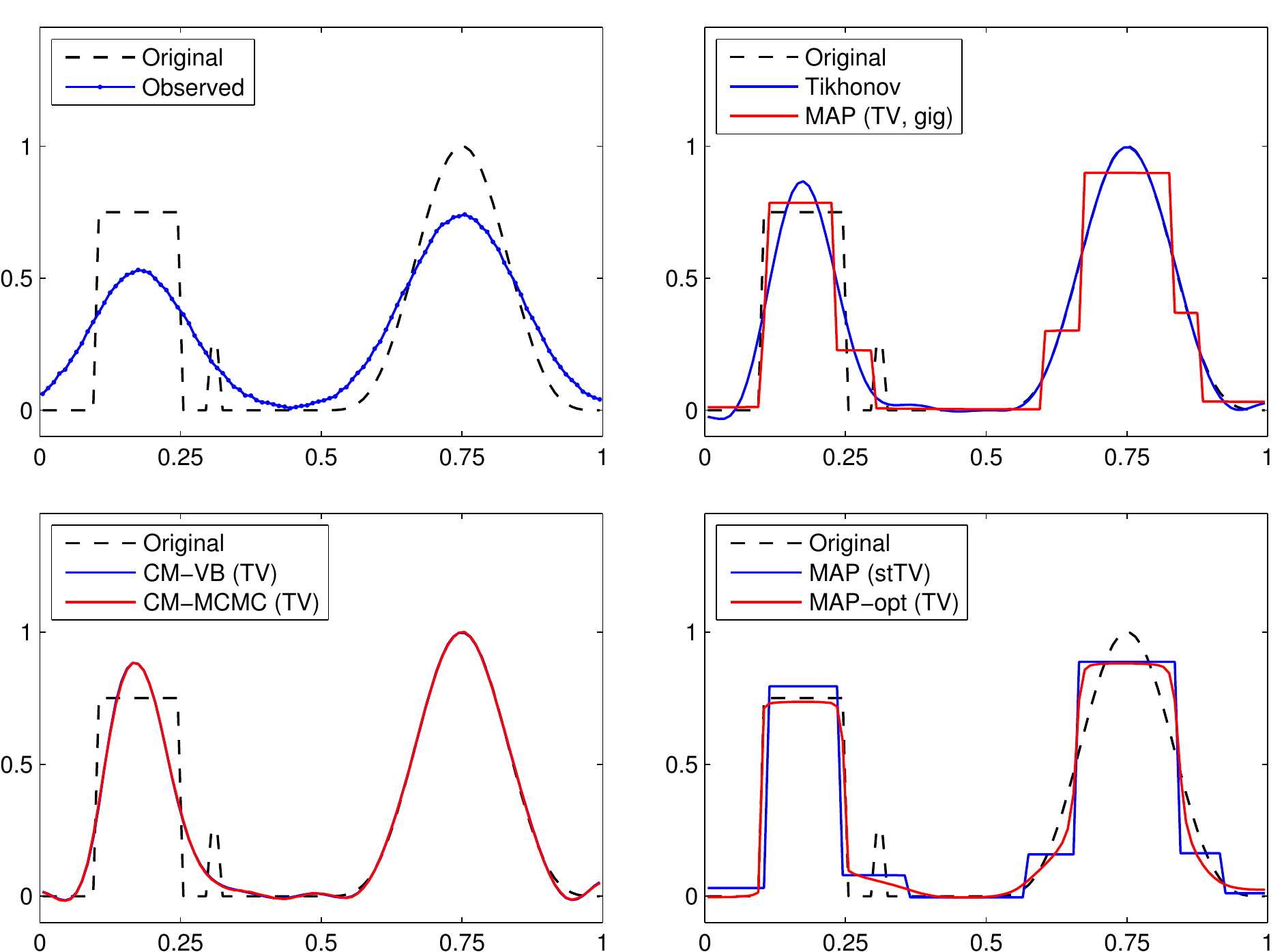}}
\caption{\label{fig:blur1d_demo1b} TV regularisation of one-dimensional partly blocky, partly smooth test image. The noise level was BSNR$=$40 dB. 
CM-VB and CM-MCMC are the CM estimates using VB and Gibbs sampler, respectively, stTV refers to the t-distribution TV prior and MAP-opt is the solution computed using deterministic optimisation approach.} 
\end{center}
\end{figure}

\begin{figure}
\begin{center}
\resizebox{.6\textwidth}{!}{\includegraphics{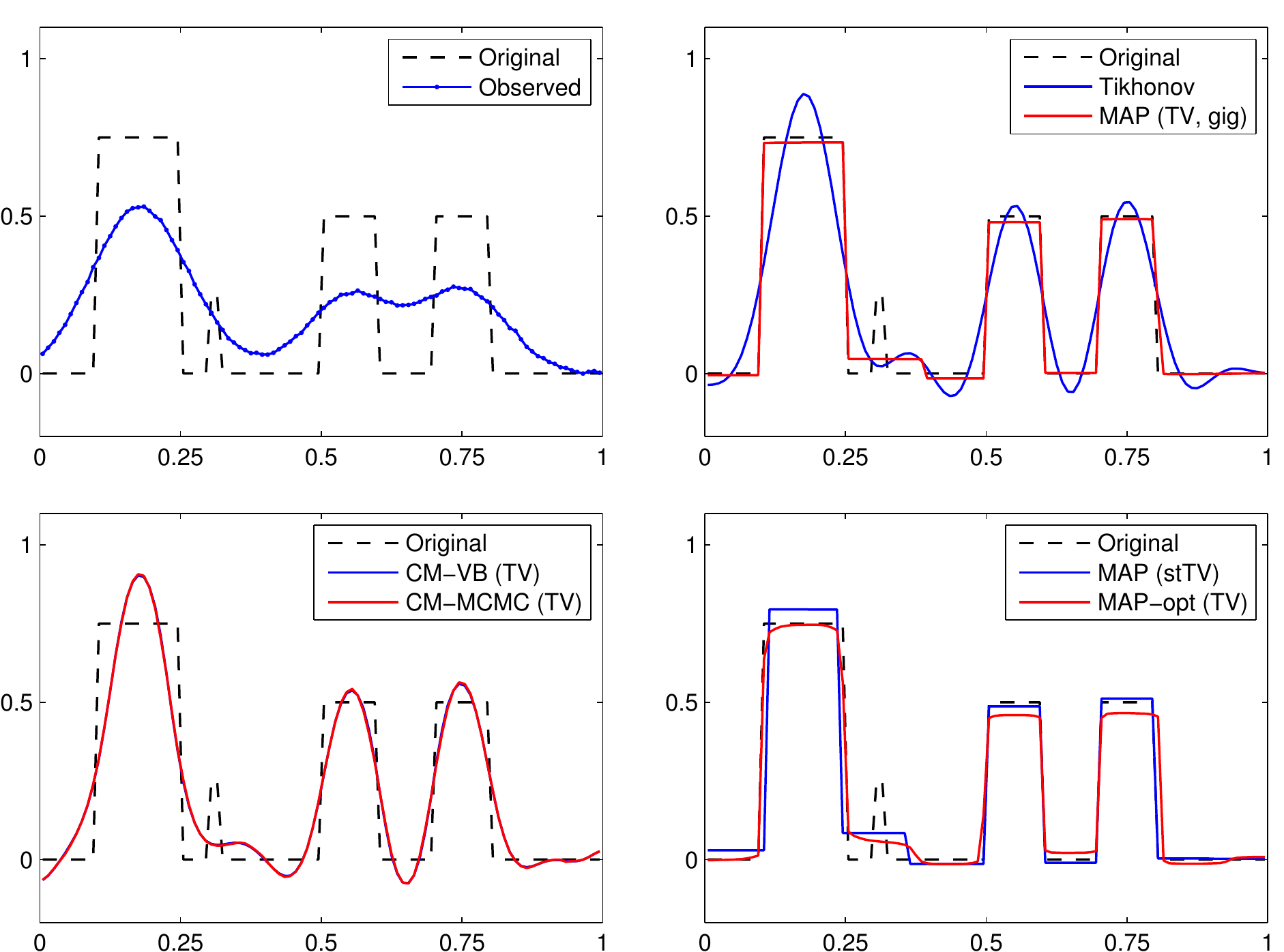}}
\caption{\label{fig:blur1d_demo2} TV regularisation of one-dimensional blocky image. Here the noise level was BSNR$=$30 dB. The algorithms used are the same as in Figure \ref{fig:blur1d_demo1b}.} 
\end{center}
\end{figure}


\begin{figure}
\begin{center}
\resizebox{0.75\textwidth}{!}{\includegraphics{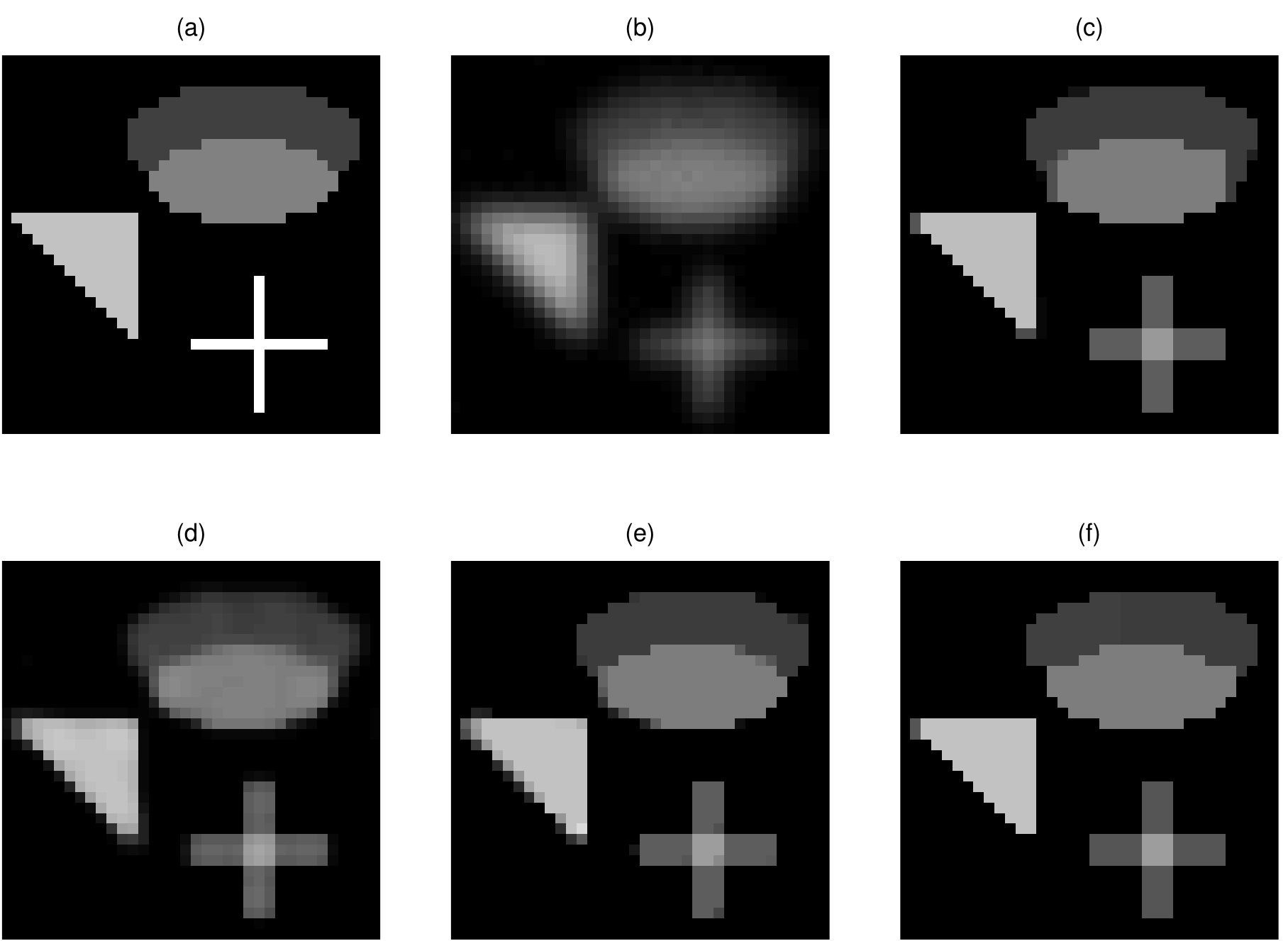}}
\caption{\label{fig:blur2d_demo} (a) The original image. (b) Blurred and noisy image. (c-e) Image reconstruction using the algorithm based on our hierarchical TV model, (c) anisotropic TV (MAP), (d) anistropic TV (CM), (e) 2-dimensional Laplace TV prior (MAP), (f) t-distribution TV prior (MAP).} 
\end{center}
\end{figure}

\begin{figure}
\begin{center}
\resizebox{0.75\textwidth}{!}{\includegraphics{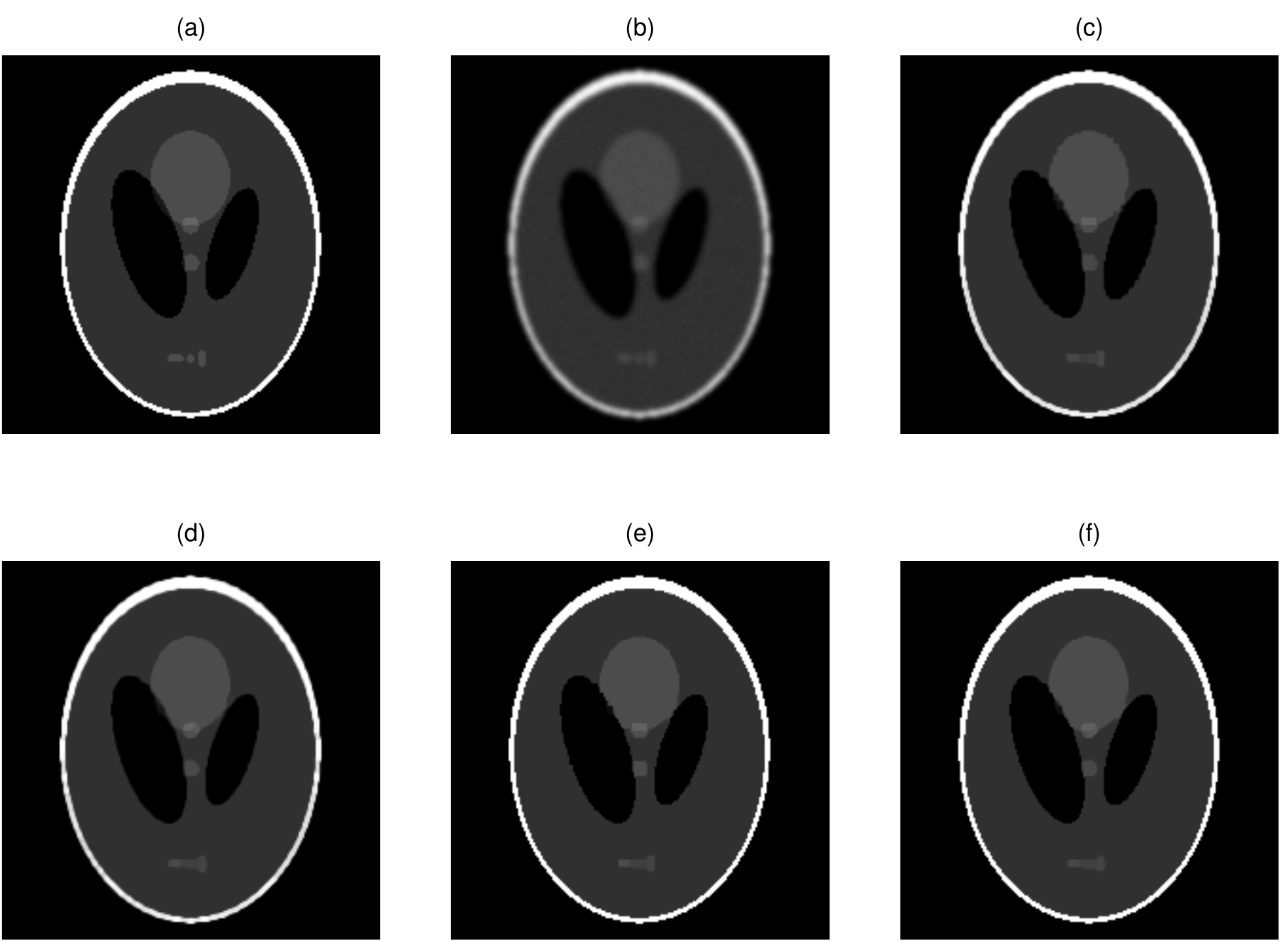}}
\caption{\label{fig:blur2d_phantom_demo} (a) The original image. (b) Blurred and noisy image. (c-e) Image reconstruction using the algorithm based on our hierarchical TV model, (c) anisotropic TV (MAP), (d) 2-dimensional Laplace TV prior (MAP), (e) t-distribution TV prior (MAP), (f) 2-dimensional t-distribution TV prior (MAP).} 
\end{center}
\end{figure}

\label{lastpage}

\end{document}